\documentclass[11pt,a4paper,reqno]{amsart}
\usepackage[english]{babel}
\usepackage[T1]{fontenc}
\usepackage{verbatim}
\usepackage{palatino}
\usepackage{amsmath}
\usepackage{mathabx}
\usepackage{amssymb}
\usepackage{amsthm}
\usepackage{amsfonts}
\usepackage{graphicx}
\usepackage{esint}
\usepackage{color}
\usepackage{mathtools}
\usepackage{overpic}

\usepackage[colorlinks = true, citecolor = black]{hyperref}
\author{Tuomas Orponen}
\title[Visibility problem]{On the dimension of visible parts}
\address{Department of Mathematics and Statistics\\ University of Helsinki,
P.O. Box 68 (Pietari Kalmin katu 5)\\
FI-00014 University of Helsinki\\
Finland}
\email{tuomas.orponen@helsinki.fi}
\date{\today}
\thanks{T.O. is supported by the Academy of Finland via the projects \emph{Quantitative rectifiability in Euclidean and non-Euclidean spaces} and \emph{Incidences on Fractals}, grant Nos. 309365, 314172, 321896. T.O. is also supported by the University of Helsinki via the project \emph{Quantitative rectifiability of sets and measures in Euclidean spaces and Heisenberg groups}, project No. 7516125.}

\newcommand{\R}{\mathbb{R}}

\newcommand{\N}{\mathbb{N}}

\newcommand{\calT}{\mathcal{T}}
\newcommand{\calL}{\mathcal{L}}
\newcommand{\calD}{\mathcal{D}}
\newcommand{\calH}{\mathcal{H}}

\newcommand{\calG}{\mathcal{G}}
\newcommand{\calF}{\mathcal{F}}

\newcommand{\calQ}{\mathcal{Q}}

\newcommand{\calM}{\mathcal{M}}

\newcommand{\spt}{\operatorname{spt}}
\newcommand{\Hd}{\dim_{\mathrm{H}}}

\newcommand{\spa}{\operatorname{span}}
\newcommand{\diam}{\operatorname{diam}}

\newcommand{\vis}{\mathrm{Vis}}

\def\Barint_#1{\mathchoice
          {\mathop{\vrule width 6pt height 3 pt depth -2.5pt
                  \kern -8pt \intop}\nolimits_{#1}}%
          {\mathop{\vrule width 5pt height 3 pt depth -2.6pt
                  \kern -6pt \intop}\nolimits_{#1}}%
          {\mathop{\vrule width 5pt height 3 pt depth -2.6pt
                  \kern -6pt \intop}\nolimits_{#1}}%
          {\mathop{\vrule width 5pt height 3 pt depth -2.6pt
                  \kern -6pt \intop}\nolimits_{#1}}}

\numberwithin{equation}{section}

\theoremstyle{plain}
\newtheorem{thm}[equation]{Theorem}
\newtheorem*{"thm"}{"Theorem"}

\newtheorem{lemma}[equation]{Lemma}

\newtheorem{proposition}[equation]{Proposition}

\theoremstyle{definition}

\theoremstyle{remark}
\newtheorem{remark}[equation]{Remark}

\newcommand{\nref}[1]{(\hyperref[#1]{#1})}

\DeclareMathSymbol{\intop}  {\mathop}{mathx}{"B3}

\begin{document}

\begin{abstract} I prove that the visible parts of a compact set in $\R^{n}$, $n \geq 2$, have Hausdorff dimension at most $n - \tfrac{1}{50n}$ from almost every direction. \end{abstract}

\maketitle

\tableofcontents

\section{Introduction}  Let $n \geq 2$, $e \in S^{n - 1}$, and let $\ell_{e} := \{te : t \geq 0\}$ be the "positive" closed half-line spanned by $e$. Let $K \subset \R^{n}$ be compact. The \emph{visible part of $K$ in direction e}, denoted $\vis_{e}(K)$, is the set of points $x \in \R^{n}$ satisfying
\begin{displaymath} (x + \ell_{e}) \cap K = \{x\}. \end{displaymath}
Alternatively, $\vis_{e}(K)$ is the set of points $x \in K$ with the property
\begin{equation}\label{form41} y \in K \quad \text{and} \quad \pi_{e}(x) = \pi_{e}(y) \quad \Longrightarrow \quad y \cdot e \leq x \cdot e. \end{equation}
Here, and in the sequel, I will write $\pi_{e} \colon \R^{n} \to e^{\perp}$ for the orthogonal projection to the $(n - 1)$-plane $e^{\perp}$. Evidently $\vis_{e}(K) \subset K$, so $\Hd \vis_{e}(K) \leq \Hd K$. Here $\Hd$ stands for Hausdorff dimension. Since $\pi_{e}(K) = \pi_{e}(\vis_{e}(K))$, it follows from the Marstrand-Mattila projection theorem \cite{MR0063439,MR0409774} (or \cite[Corollary 9.4]{zbMATH01249699}) that
\begin{displaymath} \Hd \vis_{e}(K) \geq \min\{\Hd K,n - 1\} \end{displaymath} 
for $\calH^{n - 1}$ almost every $e \in S^{n - 1}$. Does the converse inequality hold? This \emph{visibility conjecture} is a well-known open question in geometric measure theory, mentioned explicitly for example in \cite[(1.3)]{MR2928497}, \cite[Conjecture 1.3]{MR2988729}, and \cite[Problem 11]{MR2044636}. The answer is positive if $\Hd K \leq n - 1$, simply because $\vis_{e}(K) \subset K$. So, the open question concerns the case $\Hd K > n - 1$, and, explicitly, the problem is then to show that
\begin{displaymath} \Hd \vis_{e}(K) = n - 1 \quad \text{for $\calH^{n - 1}$ a.e.} \quad e \in S^{n - 1}. \end{displaymath}
To give an idea of what is involved, consider a construction of Davies and Fast \cite{MR492190} from 1978: there exists a compact set $K \subset \R^{2}$ with $\Hd K = 2$ such that $K = \vis_{e}(K)$ for a dense $G_{\delta}$-set of directions $e \in S^{1}$. In particular, $\Hd \vis_{e}(K) = 2$ for these directions $e$. It has been open, until now, if $\Hd \vis_{e}(K) = 2$ is possible for a set of directions $e \in S^{1}$ of positive measure. Theorem \ref{main} says that it is not. It has, however, been known since Marstrand's slicing theorem \cite{MR0063439} in 1954 that if $\Hd K > 1$, then $K = \vis_{e}(K)$ can only hold for a null set of directions $e \in S^{1}$. More precisely, the main result in \cite{MR3145914} shows that this is only possible for a set of directions $e \in S^{1}$ of dimension $\leq 2 - \Hd K$.

Meanwhile, a positive answer -- in fact a full solution -- to the visibility problem in $\R^{2}$ has been obtained for several classes of special sets:
\begin{itemize}
\item Quasicircles, graphs of continuous functions, and some self-similar sets \cite{MR1975783},
\item Self-similar sets (with enough separation) whose projections are intervals \cite{MR2988729},
\item Fractal percolation (almost surely) \cite{MR2928497}.
\end{itemize}
Another remarkable partial result is due to O'Neil \cite{MR2327025}: he considers a variant of the visibility problem concerning the sets $\vis_{x}(K)$ -- the visible parts of $K$ from points $x \in \R^{2} \, \setminus \, K$ (the precise definition is easy to guess, or see \cite{MR2327025}). Then, if $\Gamma \subset \R^{2}$ is a compact continuum with $s := \Hd \Gamma \geq 1$, O'Neil proves that 
\begin{displaymath} \Hd \vis_{x}(\Gamma) \leq 1/2 + \sqrt{s - 3/4} \end{displaymath}
for Lebesgue almost every viewpoint $x \in \R^{2} \, \setminus \, \Gamma$. The right hand side of O'Neil's inequality is strictly smaller than $s$ for $s \in (1,2]$ and also stays bounded away from $2$ as $s \nearrow 2$. The main caveat in O'Neil's result is that it uses the continuum hypothesis (namely the hypothesis that $\Gamma$ is a continuum, not the other continuum hypothesis!) in an essential way, and in particular does not rule out the possibility of positively many $2$-dimensional visible parts for totally disconnected sets. 

For general compact sets in $\R^{n}$ (or even $\R^{2}$), the only positive result, as far as I know, is \cite[Theorem 1.1]{MR2177426}: a special case of it implies that if $K \subset \R^{n}$ is a compact set with $0 < \calH^{s}(K) < \infty$ for $n - 1 < s \leq n$, then $\calH^{s}(\vis_{e}(K)) = 0$ for $\calH^{n - 1}$ almost every $e \in S^{n - 1}$. So, in this sense, visible parts of $K$ tend to be smaller than $K$, as soon as $\Hd K > n - 1$. The main result of this paper improves on this conclusion considerably for sets with dimension sufficiently close to $n$:

\begin{thm}\label{main} Let $K \subset \R^{n}$ be compact, $n \geq 2$. Then
\begin{displaymath} \Hd \vis_{e}(K) \leq n - \tfrac{1}{50n} \quad \text{for $\calH^{n - 1}$ a.e. } e \in S^{n - 1}. \end{displaymath} 
\end{thm}

\begin{remark} The constant $50$ is a little arbitrary, and could be slightly lowered by optimising the argument. On the other hand, it seems likely that more ideas will be needed to get an upper bound of the form $n - c$ for some absolute $c > 0$. In particular, there is a clear obstruction why the method cannot yield a universal upper bound lower than $n - \tfrac{1}{2}$. Namely, the final part of the proof of Theorem \ref{main} in Section \ref{s:badLines} deals with "bad" lines $\mathcal{L}_{b,e}$ parallel to $e \in S^{n - 1}$, whose union is denoted $L_{b,e}$. The proof has nothing to say about $L_{b,e} \cap \vis_{e}(K)$, but it will be shown, roughly speaking, that $\Hd L_{b,e} \leq 2(n - 1) + 1 - \Hd K$ for a.e. $e \in S^{n - 1}$. Consequently, if $\Hd K < n - \tfrac{1}{2}$, it may happen than $\Hd L_{b,e} > \Hd K$, and the trivial estimate $\Hd \vis_{e}(K) \leq \Hd K$ beats the best bound the proof of Theorem \ref{main} can offer.   \end{remark}

I close the section by mentioning that the "visibility problem" may refer to many distinct questions within geometric measure theory. For example: from how many viewpoints can a planar set of dimension $> 1$ be invisible? Or how to quantify the invisibility of purely $1$-unrectifiable sets? For more reading on these topics, see for example \cite{MR3458388,MR3188064,MR1762426,MR1798576,MR3526481,MR2641082,MR3778538,MR3892404,MR2865537,MR2329222}. 

\subsection{A few words on the method} The idea that visible parts should typically be at most $1$-dimensional quite likely originates from the following observation. Let $\delta \in 2^{-\N}$, let $\calD_{\delta}$ be the family of dyadic squares $Q \subset [0,1)^{2}$ of side-length $\delta$, and let $\calF \subset \calD_{\delta}$ be an arbitrary collection. Consider the union
\begin{displaymath} F := \bigcup_{Q \in \calF} Q. \end{displaymath}
Then 
\begin{equation}\label{form5} N(\vis_{(1,0)}(F),\delta) \leq \delta^{-1}. \end{equation}
Here $N(A,\delta)$ is the minimal number of balls of radius $\delta$ required to cover a (bounded) set $A$. The point is simply that whenever two squares $Q,Q' \in \calF$ lie in the same vertical "column", then the lower completely "blocks the upper from view". On the other hand, the collection of "lowest" squares in $\calF$ clearly has cardinality $\leq \delta^{-1}$.

Why does this argument not prove the whole conjecture? Assume that $K \subset F$ is a compact set such that $K \cap Q \neq \emptyset$ for all $Q \in \calF$. Then $F$ can be viewed as a "$\delta$-discretisation" of $K$. Nonetheless, \eqref{form5} implies absolutely nothing about $\vis_{(1,0)}(K)$: the visible part of $K$ can easily contain points in multiple squares of $\calF$ -- even all of them -- in any fixed vertical column. Therefore, the best universal estimate is the trivial one: $N(\vis_{(1,0)}(K),\delta) \lesssim \delta^{-2}$.

Evidently, it would be useful to know that if $Q,Q' \in \calF$ lie in the same vertical column, then $K \cap Q$ "blocks a part of $K \cap Q'$ from view". If $\Hd K > 1$ (the only interesting case), this not unreasonable: Marstrand's projection theorem \cite{MR0063439} tells us that we may expect both $\pi_{(1,0)}(K \cap Q)$ and $\pi_{(1,0)}(K \cap Q')$ to have positive length (at least if $(1,0)$ is replaced by a generic choice $e \in S^{1}$). If these positive-length sets, moreover, happen to intersect, then at least a part of $K \cap Q'$ "hides behind" $K \cap Q$.  

The main point in the proof of Theorem \ref{main} is to quantify -- even if quite weakly -- the idea above. Here is a false, but perhaps illuminating, statement: the typical $\pi_{e}$-projection of an $s$-dimensional set $K \subset \R^{2}$, with $s > 1$, not only has positive length, but actually "fills" $\spa(e)$ up to a set of dimension $2 - s < 1$. This is formally false for the reason that $K_{e} := \pi_{e}(K)$ is compact, and certainly does not fill most of $\spa(e)$. But, whenever $K_{e}$ has positive length, then any dense union of translates of $K_{e}$ fills $\R$ up to a $\calH^{1}$-null set. And if $K_{e}$ is the typical projection of an $s$-dimensional compact set, $s > 1$, then $K_{e}$ satisfies something even better:
\begin{proposition}\label{mainProp} Let $1 \leq s \leq 2$, and assume that $E \subset \R$ is a Borel set supporting a Borel probability measure $\nu$ with
\begin{equation}\label{form23} \int_{\R} |\hat{\nu}(\xi)|^{2}|\xi|^{s - 1} \, d\xi < \infty. \end{equation} 
Then any dense union of $E$ covers all of $\R$, except for a set of dimension $\leq 2 - s$. In other words, if $D \subset \R$ is dense, then $\Hd [\R \, \setminus \, (E + D)] \leq 2 - s$.
\end{proposition}
It is well-known that if $\mu$ is a finite Radon measure on $\R^{2}$ with $I_{s}(\mu) < \infty$, then $\calH^{1}$-almost every projection $\nu = \pi_{e}(\mu)$ satisfies \eqref{form23}. The proof is very simple, see \eqref{form24}. 
\begin{proof}[Proof of Proposition \ref{mainProp}] Assume to the contrary that $\Hd [\R \, \setminus \, (E + D)] > 2 - s$, and pick, using Frostman's lemma, a non-trivial Radon measure $\eta$ with
\begin{equation}\label{form25} \spt \eta \subset \R \, \setminus \, (E + D) \quad \text{and} \quad I_{2 - s}(\eta) \sim \int_{\R} |\hat{\nu}(\xi)|^{2}|\xi|^{1 - s} \, d\xi < \infty. \end{equation}
On the right, we used the well-known Fourier-representation formula for $I_{2 - s}(\nu)$, see \cite[Lemma 12.12]{zbMATH01249699}. The first point in \eqref{form25} in particular implies that if $x \in A := \spt \eta$ and $y \in E$, then there is no number $q \in D$ such that $x - y = q$. In other words,
\begin{displaymath} [A - E] \cap D = \emptyset. \end{displaymath} 
To reach a contradiction, it now suffices to argue that $A - E$ has non-empty interior. To see this, note that $A - E$ contains the support of $\rho := \eta \ast \tilde{\nu}$, where $\tilde{\nu}$ is the measure on $\R$ defined by $\tilde{\nu}(C) := \nu(-C)$. But
\begin{align*} \int |\hat{\rho}(\xi)| \, d\xi = \int_{\R} |\hat{\eta}(\xi)||\widehat{\widetilde{\nu}}(\xi)| \, d\xi \leq \left(\int_{\R} |\hat{\eta}(\xi)|^{2}|\xi|^{1 - s} \, d\xi \right)^{1/2}\left(\int_{\R} |\widehat{\nu}(\xi)|^{2}|\xi|^{s - 1} \, d\xi \right)^{1/2} < \infty \end{align*} 
by \eqref{form23}, \eqref{form25}, and Cauchy-Schwarz, so $\hat{\rho} \in L^{1}(\R)$, and hence $\rho \in C(\R)$. Therefore $\spt \rho \subset A - E$ has non-empty interior, as claimed, and the proof of the proposition is complete. \end{proof}
The idea that the projections of a $> m$-dimensional set to $m$-dimensional subspaces should typically "fill" everything except a $< m$-dimensional set is at the core of the proof of Theorem \ref{main}, and the reader will recognise a more quantitative version of the previous proof appearing in Section \ref{s:badLines}.

\subsection{Acknowledgements} I am thankful to Katrin F\"assler and Eino Rossi for helpful discussions, and to the anonymous reviewer for reading the paper carefully and, making a number of helpful suggestions.

\section{Proof of the main result}

This section contains the proof of Theorem \ref{main}. We assume with no loss of generality that $K \subset [0,1)^{n}$. We write
\begin{displaymath} \tau := \tfrac{1}{50n} \quad \text{and} \quad \varepsilon := 2\tau, \end{displaymath}
and use the variant of Frostman's lemma contained in Appendix \ref{s:frostman}, namely Lemma \ref{frostman}, to find a Radon measure $\mu$ supported on $K$, satisfying
\begin{equation}\label{form36} \mu(B(x,r)) \leq r^{n - \tau}, \qquad x \in \R^{n}, \; r > 0, \end{equation}
and
\begin{equation}\label{form35} \mu(\overline{Q}) \gtrsim \min\{\calH^{n - \tau}_{\infty}(K \cap Q), \ell(Q)^{n} \} \end{equation}
for all dyadic cubes $Q \subset [0,1)^{n}$. Write $s := n - \tfrac{1}{4}$, and note that
\begin{equation}\label{form43} n - \tfrac{1}{2} + \tau < s < n - \tau. \end{equation}
The second inequality combined with \eqref{form36} gives 
\begin{displaymath} I_{s}(\mu) = \iint \frac{d\mu(x) \, d\mu(y)}{|x - y|^{s}} < \infty. \end{displaymath}
The constant $I_{s}(\mu)$ will be regarded as "absolute" below, and the implicit constants in the "$\lesssim"$ notation are allowed to depend on it. I will also abbreviate $\lesssim_{n}$ to $\lesssim$. It might be worth remarking here that nothing prevents the possibility $\mu \equiv 0$: this is, in fact, the case if $\calH^{n - \tau}(K) = 0$, in which case the statement of the theorem simply follows from $\vis_{e}(K) \subset K$.

Next, fix a dyadic scale $\delta \in (0,\tfrac{1}{100})$. Assume with no loss of generality that $\delta^{\varepsilon} \in 2^{-\N}$ (one may restrict attention to those $\delta = 2^{-N} > 0$ such that $N/(50n) \in \N$). The scale $\delta > 0$ may be taken arbitrarily small to begin with, and I will often do so (to negate the effect of certain multiplicative constants) without further mention. Let 
\begin{displaymath} \calQ := \{Q \in \calD_{\delta^{\varepsilon}} : Q \cap K \neq \emptyset\}, \end{displaymath}
where, in general, $\calD_{\eta}$ stands for dyadic sub-cubes $Q \subset [0,1)^{n}$ of side-length $\ell(Q) = \eta \in 2^{-\N}$. Evidently $|\calQ| \leq |\calD_{\delta^{\varepsilon}}| = \delta^{-n\varepsilon}$. For $Q \in \calQ$, let $\mu_{Q} := \mu|_{Q}$. Then of course $I_{s}(\mu_{Q}) \leq I_{s}(\mu)$ for all $Q \in \calQ$, and consequently
\begin{equation}\label{form24} \int_{S^{n - 1}} \int_{e^{\perp}} |\widehat{\mu_{Q}}(\xi)|^{2}|\xi|^{s - (n - 1)} \, d\calH^{n - 1}(\xi) \, d\calH^{1}(e) \sim \int_{\R^{n}} |\widehat{\mu_{Q}}(\xi)|^{2}|\xi|^{s - n} \, d\xi \lesssim 1, \end{equation}
using (generalised) integration in polar coordinates in the first step, see \cite[(24.2)]{MR3617376}, and the well-known Fourier-representation \cite[Lemma 12.12]{zbMATH01249699} for the $s$-energy in the second step. In particular, the "exceptional set"
\begin{displaymath} E_{Q} := \left\{e \in S^{n - 1} : \int_{\R} |\widehat{\mu_{Q}}(r e)|^{2}|r|^{s - 1} \, dr \geq \delta^{-\varepsilon(n + 1)} \right\} \end{displaymath}
has measure $\calH^{n - 1}(E_{Q}) \lesssim \delta^{\varepsilon(n + 1)}$. Noting again that $|\calQ| \lesssim \delta^{-\varepsilon n}$, we conclude that the "total" exceptional set
\begin{displaymath} E := \bigcup_{Q \in \calQ} E_{Q} \end{displaymath}
has length
\begin{equation}\label{form6} \calH^{n - 1}(E) \lesssim \delta^{\varepsilon}. \end{equation}
We now fix $e \in S^{n - 1} \, \setminus \, E$, and claim that
\begin{equation}\label{form7} \calH^{n - \tau}_{\infty}(\vis_{e}(K)) \lesssim \delta^{\varepsilon/2}. \end{equation}  
Before starting the proof, let us briefly observe that the theorem follows immediately from a combination of \eqref{form6} and \eqref{form7}. Namely, the exceptional set $E$ evidently depends on $\delta$, so it would be more accurate to write $E = E(\delta)$. Then, if \eqref{form6} is applied with each $\delta = 2^{-j}$ (small enough), one infers from the Borel-Cantelli lemma that $\calH^{n - 1}$ almost every point $e \in S^{n - 1}$ only lies in finitely many sets $E(2^{-j})$. The remaining points $e \in S^{n - 1}$ satisfy \eqref{form7} for all $\delta = 2^{-j}$ large enough, which in particular implies that
\begin{displaymath} \calH^{n - \tau}_{2^{-\varepsilon j/(2n)}}(\vis_{e}(K)) \lesssim 2^{-\varepsilon j/2} \end{displaymath} 
for all $j \in \N$ large enough (noting that every $\infty$-cover satisfying \eqref{form7} must in fact be a $\delta^{\varepsilon/(2n)}$-cover). In particular, the sequence $\{\calH^{n - \tau}_{2^{-\varepsilon j/(2n)}}(\vis_{e}(K))\}_{j \in \N}$ remains bounded as $j \to \infty$, and it follows that $\Hd \vis_{e}(K) \leq n - \tau$, as claimed.

To get started with \eqref{form7}, for each $Q \in \calQ$, let 
\begin{displaymath} \calQ'_{\delta}(Q) := \{Q_{\delta} \in \calD_{\delta} : Q_{\delta} \subset Q \text{ and } Q \cap K \neq \emptyset\}. \end{displaymath} 
We also write $\calQ'_{\delta}$ for the union of the collections $\calQ'_{\delta}(Q)$, over all $Q \in \calQ$. Thus $\calQ_{\delta}'$ is a cover for $K$. A little technical annoyance is that some cubes in $\calQ'_{\delta}$ may perhaps be \emph{light}, i.e. satisfy $\mu(Q_{\delta}) \leq \delta^{n + \varepsilon}$. (In fact, if $\calH^{n - \tau}(K) = 0$, then all cubes in $\calQ_{\delta}'$ will be light by \eqref{form36}, but in that case there is nothing to prove anyway.) Such cubes turn out to be undesirable, and we wish to get rid of them immediately. The lower bound \eqref{form35} gets used here: for any light cube (if $\delta > 0$ is small enough), evidently
\begin{displaymath} \calH^{n - \tau}_{\infty}(K \cap Q) \lesssim \mu(\overline{Q_{\delta}}) = \mu(Q_{\delta}) \leq \delta^{n + \varepsilon}. \end{displaymath} 
The middle equation follows from the fact that $\mu$ charges no lines by \eqref{form36}. Therefore, if $K_{\mathrm{light}}$ is the part of $K$ contained in the union of the light cubes, we have
\begin{displaymath} \calH^{n - \tau}_{\infty}(\vis_{e}(K) \cap K_{\mathrm{light}}) \leq \calH^{n - \tau}_{\infty}(K_{\mathrm{light}}) \lesssim \delta^{\varepsilon}, \end{displaymath}
and this is even better than \eqref{form7}. Thus, \eqref{form7} will follow once we manage to show that
\begin{equation}\label{form37} \calH^{n - \tau}_{\infty}(\vis_{e}(K) \cap K_{\mathrm{h}}) \lesssim \delta^{\varepsilon/2}, \end{equation}
where $K_{\mathrm{h}} := K \, \setminus \, K_{\mathrm{light}}$ -- i.e. the part of $K$ contained in the union of the "heavy" cubes $\calQ_{\delta} := \{Q_{\delta} \in \calQ_{\delta}' : Q_{\delta} \text{ is not light}\}$ (define also $\calQ_{\delta}(Q) := \calQ_{\delta} \cap \calQ_{\delta}'(Q)$ for $Q \in \calQ$). The upshot of the previous discussion is then that
\begin{equation}\label{form22} \mu(Q_{\delta}) \geq \delta^{n + \varepsilon}, \qquad Q_{\delta} \in \calQ_{\delta}, \end{equation} 
and $\calQ_{\delta}$ is a cover for $K_{\mathrm{h}}$. As a small digression, I point out that $\vis_{e}(K) \cap K_{\mathrm{h}}$ may be a strict subset of $\vis_{e}(K_{\mathrm{h}})$ (as points in $K_{\mathrm{h}}$ may be occasionally be "blocked from view" by points in $K_{\mathrm{light}}$).

Denote the lines parallel to $e$ by $\calL$. We split them into two disjoint sub-families, the "good" lines $\calL_{g}$ and the "bad" lines $\calL_{b}$. Informally, the lines $\ell \in \calL_{b}$ intersect a high "stack" of $\delta$-cubes in some collection $\calQ_{\delta}(Q)$, $Q \in \calQ$, but still manage to "percolate" through $K \cap Q$. More precisely, we define that $\ell \in \calL_{b}$ if there exists $Q \in \calQ$ such that
\begin{equation}\label{form10} |\{Q_{\delta} \in \calQ_{\delta}(Q) : Q_{\delta} \cap \ell(2\delta) \neq \emptyset\}| \geq \delta^{2\varepsilon - 1} \quad \text{and} \quad \ell \cap K \cap \overline{Q} = \emptyset, \end{equation} 
see Figure \ref{fig1}.
\begin{center}
\begin{figure}[h!]
\begin{overpic}[scale = 0.40]{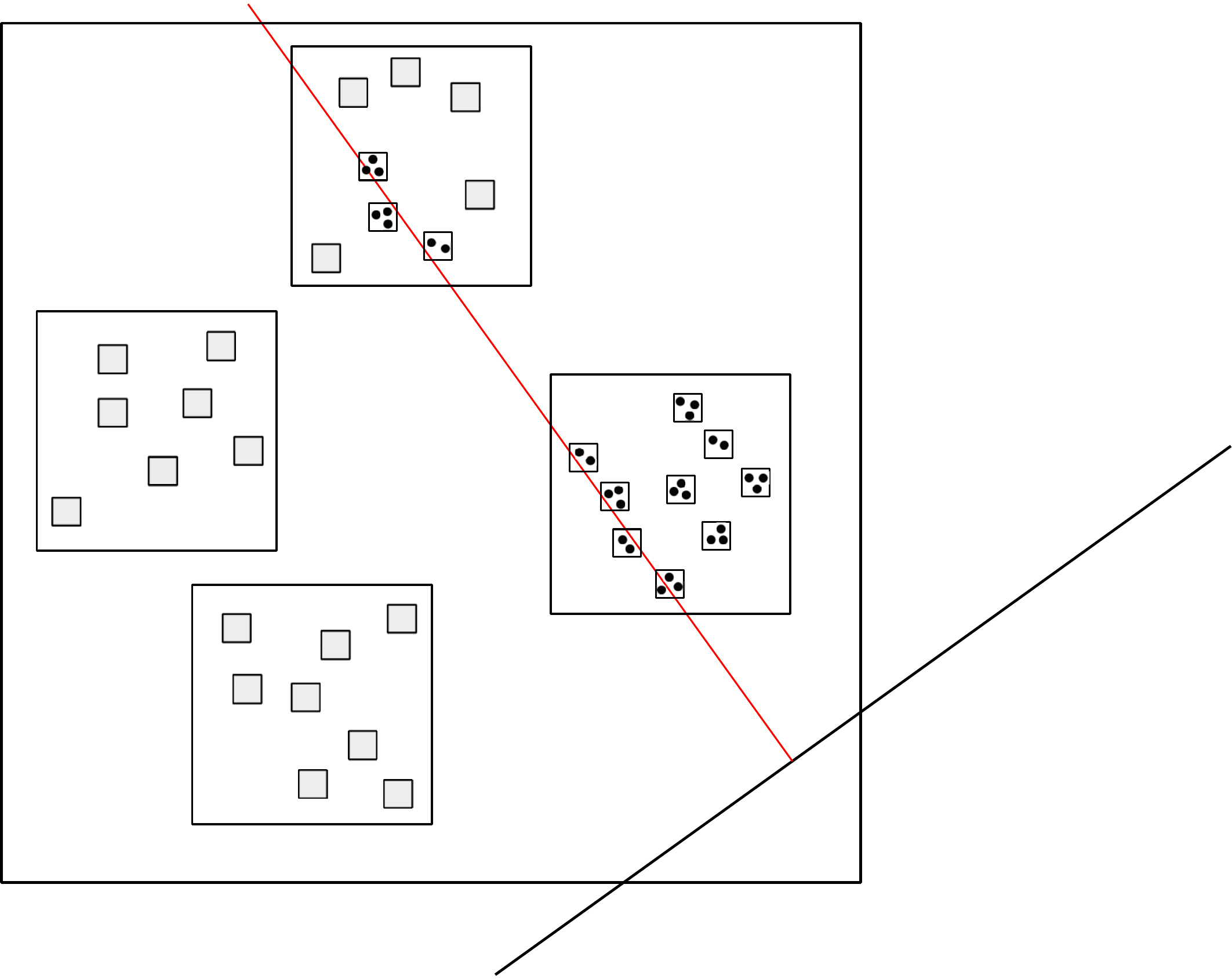}
\put(45,45.5){$Q$}
\put(55,22){$\ell$}
\put(75,32){$e^{\perp}$}
\end{overpic}
\caption{The red line $\ell$ is in $\mathcal{L}_{b}$: it hits many cubes in $\calQ_{\delta}(Q)$ but not $K \cap Q$.}\label{fig1}
\end{figure}
\end{center}
At the risk of over-explaining, I emphasise that the cubes $Q_{\delta} \in \calQ_{\delta}(Q)$ are heavy.  Also define $\calL_{g} := \calL \, \setminus \, \calL_{b}$, and set
\begin{displaymath} L_{b} := \bigcup_{\ell \in \mathcal{L}_{b}} \ell \quad \text{and} \quad L_{g} := \bigcup_{\ell \in \mathcal{L}_{g}} \ell. \end{displaymath} 
The proof of \eqref{form37} now splits into separate estimates for  $L_{g} \cap \vis_{e}(K) \cap K_{\mathrm{h}}$ and $L_{b} \cap \vis_{e}(K) \cap K_{\mathrm{h}}$. 
\subsection{Visible part on the good lines} Sub-divide $[0,1)^{n} \supset K$ into $\sim \delta^{-(n - 1)}$ tubes $\calT_{\delta}$ of width $\delta$ which are perpendicular to $e$. We claim that
\begin{equation}\label{form9} N(\vis_{e}(K) \cap K_{\mathrm{h}} \cap L_{g} \cap T,\delta) \lesssim \delta^{\varepsilon - 1}, \qquad T \in \calT_{\delta}. \end{equation}
This will immediately yield
\begin{displaymath} \calH^{n - \tau}_{\delta}(\vis_{e}(K) \cap K_{\mathrm{h}} \cap L_{g}) \lesssim |\calT_{\delta}| \cdot \delta^{\varepsilon - 1} \cdot \delta^{n - \tau} \lesssim \delta^{\varepsilon/2}, \end{displaymath} 
recalling that $\varepsilon = 2\tau$, and this estimate is better than \eqref{form7}.

To prove \eqref{form9}, fix $T \in \calT_{\delta}$. There are two options. First, it may happen that
\begin{displaymath} |\{Q_{\delta} \in \calQ_{\delta}(Q) : Q_{\delta} \cap T \neq \emptyset\}| < \delta^{2\varepsilon - 1}, \qquad Q \in \calQ, \end{displaymath}
or in other words $T$ never meets a "high stack" of heavy $\delta$-cubes in any single collection $\calQ_{\delta}(Q)$. In this case simply
\begin{displaymath} N(\vis_{e}(K) \cap K_{\mathrm{h}} \cap L_{g} \cap T,\delta) \leq N(K_{\mathrm{h}}  \cap T,\delta) \lesssim \delta^{-\varepsilon} \cdot \delta^{2\varepsilon - 1} \lesssim \delta^{\varepsilon - 1}, \end{displaymath}
recalling that the cubes in $\calQ_{\delta}$ form a cover for $K_{\mathrm{h}}$, and $T$ can meet only meet $\lesssim \delta^{-\varepsilon}$ cubes $Q \in \calQ$. This is \eqref{form9}. The other alternative is where there exists at least one $Q \in \calQ$ such that
\begin{equation}\label{form11} |\{Q_{\delta} \in \calQ_{\delta}(Q) : Q_{\delta} \cap T \neq \emptyset\}| \geq \delta^{2\varepsilon - 1}. \end{equation} 
In particular, we may choose "the $e$-highest" $Q_{1} \in \calQ$ satisfying \eqref{form11}: more precisely, let $Q_{1}$ be the cube $Q \in \calQ$ satisfying \eqref{form11} such that $\inf \{x \cdot e : x \in Q\}$ is maximised (if there are several candidates, pick any of them). Now, every line $\ell \subset T$ evidently satisfies $T \subset \ell(2\delta)$, hence
\begin{displaymath} |\{Q_{\delta} \in \calQ_{\delta}(Q_{1}) : Q_{\delta} \cap \ell(2\delta) \neq \emptyset\}| \geq \delta^{2\varepsilon - 1}, \end{displaymath} 
and consequently, by definition of $\mathcal{L}_{g}$,
\begin{equation}\label{form2} \ell \in \mathcal{L}_{g} \text{ and } \ell \subset T \quad \Longrightarrow \quad \ell \cap K \cap \overline{Q}_{1} \neq \emptyset. \end{equation}
Therefore, if $Q \in \calQ$ is another cube "lower" than $Q_{1}$, now in the precise sense
\begin{equation}\label{form17} \sup_{x \in Q} x \cdot e < \inf_{y \in Q_{1}} y \cdot e, \end{equation}
we claim that $\vis_{e}(K) \cap L_{g} \cap Q \cap T = \emptyset$ (the set $K_{\mathrm{h}}$ plays no role here), so in particular
\begin{equation}\label{form19} N(\vis_{e}(K) \cap L_{g} \cap Q \cap T,\delta) = 0. \end{equation}
Indeed, a hypothetical point $x \in \vis_{e}(K) \cap L_{g} \cap Q \cap T$ would lie on some line $\ell \in \mathcal{L}_{g}$ contained in $T$, which, by \eqref{form2}, satisfies $\ell \cap K \cap \overline{Q}_{1} \neq \emptyset$. This, and \eqref{form17}, means that some point $y \in \ell \cap K \cap \overline{Q}_{1}$ has $\pi_{e}(x) = \pi_{e}(y)$ and $x \cdot e < y \cdot e$, and hence $x \notin \vis_{e}(K)$ (recall the characterisation \eqref{form41} of $\vis_{e}(K)$).

Therefore, $\vis_{e}(K) \cap L_{g} \cap T$ is contained in the union of the cubes $Q \in \calQ$ intersecting $T$ and satisfying the converse of \eqref{form17}, that is,
\begin{equation}\label{form18} \sup_{x \in Q} x \cdot e \geq \inf_{y \in Q_{1}} y \cdot e. \end{equation}
We still need to split these cubes into two groups. First come those cubes $Q \in \calQ$ which satisfy \eqref{form18} and meet the $\delta$-tube $T$, but for which
\begin{equation}\label{form42} \inf_{x \in Q} x \cdot e \leq \inf_{y \in Q_{1}} y \cdot e. \end{equation}
Evidently there are $\sim 1$ such cubes $Q \in \calQ$ (they notably include $Q_{1}$), and for each of them we use the trivial estimate
\begin{equation}\label{form20} N(\vis_{e}(K) \cap L_{g} \cap Q \cap T,\delta) \leq N(Q \cap T,\delta) \lesssim \delta^{\varepsilon - 1}. \end{equation}
Finally, to treat the remaining cubes $Q \in \calQ$ -- which meet $T$ and satisfy the opposite of \eqref{form42} -- we recall the choice of $Q_{1}$ as the "$e$-highest" cube in $\calQ$ to satisfy \eqref{form11}. Therefore \eqref{form11} fails for the remaining $Q \in \calQ$, as specified above, and hence they satisfy
\begin{align} N(\vis_{e}(K) \cap K_{\mathrm{h}} \cap L_{g} \cap Q \cap T,\delta) & \leq N(K_{\mathrm{h}} \cap L_{g} \cap Q \cap T,\delta) \notag\\
&\label{form21} \lesssim |\{Q_{\delta} \in \calQ_{\delta}(Q) : Q_{\delta} \cap T \neq \emptyset\}| \lesssim \delta^{2\varepsilon - 1}. \end{align}
The number of cubes $Q \in \calQ$ of this type is $\lesssim \delta^{-\varepsilon}$ (just using the trivial estimate that the $\delta$-tube $T$ only meets $\lesssim \delta^{-\varepsilon}$ cubes in $\calQ$). Putting \eqref{form19}, \eqref{form20}, and \eqref{form21} together, we find that
\begin{displaymath} N(\vis_{e}(K) \cap K_{\mathrm{h}} \cap L_{g} \cap T,\delta) \lesssim \delta^{-\varepsilon} \cdot \delta^{2\varepsilon - 1} + \delta^{\varepsilon - 1} \lesssim \delta^{\varepsilon - 1}. \end{displaymath} 
This concludes the proof of \eqref{form9}.

\subsection{Visible part on the bad lines}\label{s:badLines} To complete the proof of \eqref{form37}, and hence \eqref{form7}, it remains to consider the set $\vis_{e}(K) \cap K_{\mathrm{h}} \cap L_{b}$. A very crude estimate will be made here: since $\vis_{e}(K) \cap K_{\mathrm{h}} \cap L_{b} \subset L_{b} \cap [0,1)^{n}$, it suffices to show that
\begin{equation}\label{form16} \calH^{n - \tau}_{\infty}(L_{b} \cap [0,1)^{n}) \lesssim \delta^{1/8}. \end{equation}
To prove \eqref{form16}, we split the lines in $\calL_{b}$ into the natural subsets $\calL_{Q,b}$ associated to individual balls $Q \in \calQ$: we write $\ell \in \calL_{Q,b}$ if the badness condition \eqref{form10} of $\ell$ is satisfied for $Q$. The sets $\calL_{Q,b}$ need not be disjoint, but this is irrelevant: since $|\calQ| \leq \delta^{-\varepsilon n}$, it suffices to prove that
\begin{displaymath} \calH^{n - \tau}_{\infty}([0,1)^{n} \cap L_{Q,b}) \lesssim \delta^{1/8 + \varepsilon n}, \end{displaymath} 
where $L_{Q,b}$ is the union of the lines in $\calL_{Q,b}$, and then sum up the estimates to arrive at \eqref{form16}. Moreover, since $\varepsilon \leq 1/(8n)$, the preceding displayed estimate will clearly follow from
\begin{equation}\label{form12} \calH_{\infty}^{n - 1 - \tau}(\pi_{e}(L_{Q,b})) \leq \delta^{1/4}. \end{equation} 
Assume that \eqref{form12} fails, and write $H := H_{Q,e} := \pi_{e}(L_{Q,b}) \subset [-n,n]$ (all the bad lines of course need to meet $[0,1)^{n}$, and we identify the plane $e^{\perp}$ with $\R^{n - 1}$ in the sequel). Let $\nu$ be a Borel probability measure supported on $H$ satisfying
\begin{equation}\label{form39} \nu(B(x,r)) \lesssim \delta^{-1/4}r^{n - 1 - \tau}. \end{equation} 
For this, use Frostman's lemma, see \cite[Theorem 8.8]{zbMATH01249699}, and in particular the sharp version that the "best multiplicative constant" of $\nu$ is comparable to the inverse of the Hausdorff content of $H$. The same estimate alternatively follows from Lemma \ref{frostman}, applying the lower bound \eqref{lowerBound} at unit scale, and then re-normalising so that a probability measure is obtained.

Recalling from \eqref{form43} that $s > n - \tfrac{1}{2} + \tau$, we have $n - 1 - \tau > \tfrac{1}{2} + 2(n - 1) - s$, and consequently
\begin{equation}\label{form14} I_{1/2 + 2(n - 1) - s}(\nu) \lesssim \delta^{-1/4} \end{equation}
by \eqref{form39}, see \cite[p. 109]{zbMATH01249699} for this standard calculation. Since every line in $\calL_{Q,b}$ misses $K \cap \overline{Q} \supset \spt \mu_{Q}$ by definition, and $\nu$ is supported on the $\pi_{e}$-projection of these lines, we have 
\begin{displaymath} \spt \mu_{Q,e} \cap \spt \nu = \emptyset, \end{displaymath} 
where $\mu_{Q,e} := \pi_{e}(\mu_{Q})$. Both sets $\mu_{Q,e}$ and $\spt \nu$ are compact, so also their $\eta$-neighbourhoods are disjoint for $0 < \eta \ll \delta$ small enough, and hence
\begin{equation}\label{form13} 0 = \int \mu_{Q,e} \ast \varphi_{\eta} \, d\nu = \int_{\R^{n - 1}} \hat{\varphi}(\eta \xi) \widehat{\mu_{Q,e}}(\xi) \bar{\hat{\nu}}(\xi) \, d\xi. \end{equation}
Here $\varphi_{\eta}(x) = \eta^{-(n - 1)}\varphi(x/\eta)$, where $\varphi$ is any standard bump function on $\R^{n - 1}$ (smooth, non-negative, compactly supported, integral one, and $\varphi(0) > 0$). We remind the reader here that $e \in S^{n - 1} \, \setminus \, E$, so in particular $e \notin E_{Q}$, which meant that
\begin{equation}\label{form3} \int_{\R^{n - 1}} |\widehat{\mu_{Q,e}}(\xi)|^{2}|\xi|^{s - (n - 1)} \, d\xi \leq \delta^{-\varepsilon(n + 1)}. \end{equation}
(This also used the standard fact, see \cite[(5.15)]{MR3617376}, that the Fourier transform of the projected measure $\mu_{Q,e} = \pi_{e}(\mu_{Q})$ coincides with the restriction of $\widehat{\mu_{Q}}$ to the subspace $e^{\perp}$.) Now, we estimate the right hand side of \eqref{form13} as follows:
\begin{align*} \eqref{form13} & \geq \left| \int_{\R^{n - 1}} \widehat{\varphi}(C\delta \xi) \widehat{\varphi}(\eta \xi) \widehat{\mu_{Q,e}}(\xi) \bar{\hat{\nu}}(\xi) \, d\xi \right|\\
& - \left| \int_{\R^{n - 1}} [1 - \widehat{\varphi}(C\delta \xi)]\widehat{\varphi}(\eta \xi)\widehat{\mu_{Q,e}}(\xi)\bar{\hat{\nu}}(\xi) \, d\xi \right| =: I_{1} - I_{2}. \end{align*} 
Here $C \geq 1$ is an absolute constant to be chosen momentarily. We plan to estimate $I_{1}$ from below and $I_{2}$ from above, and show that in fact $I_{1} > I_{2}$ (for $\delta > 0$ small enough). To estimate $I_{2}$ from above, note that $\widehat{\varphi}$ is a bounded Lipschitz function with $\widehat{\varphi}(0) = 1$, so
\begin{displaymath} |1 - \widehat{\varphi}(C\delta \xi)| = |\widehat{\varphi}(0) - \widehat{\varphi}(C\delta \xi)| \lesssim \min\{|\delta \xi|,1\} \leq \delta^{1/4}|\xi|^{1/4}. \end{displaymath} 
Consequently, using also Cauchy-Schwarz, \eqref{form14} and \eqref{form3},
\begin{align} I_{2} & \lesssim \delta^{1/4} \int_{\R^{n - 1}} |\xi|^{1/4} |\widehat{\mu_{Q,e}}(\xi)||\bar{\hat{\nu}}(\xi)| \, d\xi \notag\\
& \leq \delta^{1/4} \left( \int_{\R^{n - 1}} |\widehat{\mu_{Q,e}}(\xi)|^{2}|\xi|^{s - (n - 1)} \, d\xi \right)^{1/2}\left(\int_{\R^{n - 1}} |\hat{\nu}(\xi)|^{2}|\xi|^{(1/2 + 2(n - 1) - s) - (n - 1)} \, d\xi \right)^{1/2} \notag\\
&\label{form15} \leq \delta^{1/4} \cdot \delta^{-\varepsilon(n + 1)/2} \cdot \delta^{-1/8} \leq \delta^{1/16}, \end{align} 
since $\varepsilon(n + 1)/2 = (n + 1)/(100n) \leq 1/16$. Finally, to obtain a lower bound for $I_{1}$, we use Parseval (again):
\begin{displaymath} I_{1} = \int [\varphi_{C\delta} \ast \varphi_{\eta} \ast \mu_{Q,e}(r)] \, d\nu(r). \end{displaymath} 
Now recall that $\nu$ was a probability measure supported on the set $H = \pi_{e}(L_{Q,b})$. By definition, if $r \in H$, then $\ell := \pi_{e}^{-1}\{r\} \in \calL_{Q,b}$, which in particular means that $\ell(2\delta)$ meets $\geq \delta^{2\varepsilon - 1}$ cubes $Q_{\delta} \in \calQ_{\delta}(Q)$. All of these cubes are heavy, and contained in $Q$, and hence satisfy $\mu_{Q}(Q_{\delta}) = \mu(Q_{\delta}) \geq \delta^{n + \varepsilon}$ (recalling \eqref{form22}). It follows that 
\begin{displaymath} \mu_{Q}(\ell(2\delta)) \gtrsim \delta^{2\varepsilon - 1} \cdot \delta^{n + \varepsilon} \geq \delta^{3\varepsilon + (n - 1)}, \end{displaymath} 
which easily implies that
\begin{displaymath} \psi_{C\delta} \ast \varphi_{\eta} \ast \mu_{Q,e}(r) \gtrsim \delta^{3\varepsilon}, \end{displaymath}
if $C \geq 1$ is chosen sufficiently large, and $\eta > 0$ sufficiently small (note that $\eta > 0$ can be taken arbitrarily small, even in a manner depending on $\delta$). Consequently $I_{1} \gtrsim \delta^{3\varepsilon}$, using that $\nu$ is a probability measure. Since $3\varepsilon = 3/(25n) \leq 3/50 < 1/16$ for $n \geq 2$, we see from this estimate and \eqref{form15} that $I_{1} - I_{2} > 0$, for all $\delta > 0$ sufficiently small. This contradicts \eqref{form13} and concludes the proof of \eqref{form16} -- and also completes the proof of Theorem \ref{main}.

\appendix

\section{Frostman's lemma with lower bounds}\label{s:frostman}

Frostman's lemma \cite[Theorem 8.8]{zbMATH01249699} states that if $E \subset \R^{n}$ is a compact set with $\calH^{s}_{\infty}(E) > 0$, then $K$ supports a measure $\mu$ satisfying the growth bound $\mu(B(x,r)) \lesssim_{n} r^{s}$, and with total variation $\|\mu\| \geq \calH^{s}_{\infty}(E)$. One might hope to improve the lower bound to $\mu(B(x,r)) \gtrsim \calH_{\infty}^{s}(B(x,r) \cap E)$ for all $x \in E$ and $r > 0$, but I do not know if this is true (and it frankly sounds a little too optimistic). The next lemma gives a weaker substitute, which turns out to be good enough for the application in this paper:

\begin{lemma}\label{frostman} Let $0 \leq s \leq n$, and let $E \subset [0,1)^{n}$ compact. Then, there exists a Radon measure $\mu$ supported on $E$, and satisfying
\begin{equation}\label{upperBound} \mu(B(x,r)) \lesssim_{n} r^{s}, \qquad x \in \R^{n}, \; r > 0, \end{equation}
and
\begin{equation}\label{lowerBound} \mu(\overline{Q}) \gtrsim_{n} \min\{\calH^{s}_{\infty}(E \cap Q),|Q|\} \end{equation}
for all dyadic cubes $Q \subset [0,1)^{n}$, where $|\cdot|$ stands for Lebesgue measure.  \end{lemma}

\begin{proof} We may assume that $\mathcal{H}^{s}_{\infty}(E) > 0$, since otherwise the measure $\mu = 0$ works. We follow the standard proof of Frostman's lemma with minor modifications to achieve the lower bound \eqref{lowerBound}. Let $\delta \in 2^{-\N}$, and let $\calD_{\delta}$ be the collection of dyadic cubes of side-length $\ell(Q) = \delta$ which are contained in $[0,1)^{n}$. Also, let $\calD_{\delta}(E) := \{Q \in \calD_{\delta} : Q \cap E \neq \emptyset\}$, and write
\begin{displaymath} E_{\delta} := \bigcup_{Q \in \calD_{\delta}(E)} Q \subset [0,1)^{n}. \end{displaymath}
We first construct a measure $\mu_{\delta} \in \calM(\overline{E_{\delta}})$, and satisfying \eqref{upperBound}-\eqref{lowerBound} for scales $\delta \leq r \leq 1$. For $Q \in \calD_{\delta}$, start by finding a measure $\mu_{\delta}^{0} \in \calM(\overline{E_{\delta}})$ such that
\begin{displaymath} \mu_{\delta}^{0}(Q) := \begin{cases} \ell(Q)^{s}, & \text{if } Q \in \calD_{\delta}(E), \\ 0, & \text{if } Q \in \calD_{\delta} \, \setminus \, \calD_{\delta}(E). \end{cases} \end{displaymath}
To be more precise, for $Q \in \calD_{\delta}$, let $\mu_{\delta}^{0}|_{Q}$ be a weighted copy of Lebesgue measure on $Q$, with weights determined by the equation above.

Assume that $\mu^{k}_{\delta}$ has already been defined for some $k \geq 0$, and consider a cube $Q \in \calD_{2^{k + 1}\delta}$. If 
\begin{displaymath} \mu_{\delta}^{k}(Q) \leq \ell(Q)^{s} = (2^{k + 1}\delta)^{s}, \end{displaymath}
set
\begin{displaymath} \mu_{\delta}^{k + 1}|_{Q} := \mu_{\delta}^{k}|_{Q}. \end{displaymath}
If, on the other hand,
\begin{equation}\label{form74} \mu_{\delta}^{k}(Q) > \ell(Q)^{s}, \end{equation}
define $\mu^{k + 1}_{\delta}|_{Q}$ as follows. Consider the (possibly empty) family $\calG := \calG_{Q}^{k}$ of maximal dyadic sub-cubes $Q' \subset Q$ of side-length $\delta \leq \ell(Q') \leq \ell(Q)$ such that
\begin{displaymath} \mu_{\delta}^{k}(Q') \leq |Q'|/2. \end{displaymath} 
The cubes in $\calG$ are disjoint, by maximality, and their union $G := \cup \calG$ satisfies
\begin{displaymath} \mu_{\delta}^{k}(G) \leq \sum_{Q' \in \calG} \frac{|Q'|}{2} \leq \frac{|Q|}{2} \leq \frac{\ell(Q)^{s}}{2}. \end{displaymath}
Then, write $B := B_{Q}^{k} := Q \, \setminus \, G \subset Q$, and define
\begin{equation}\label{form30} \mu_{\delta}^{k + 1}|_{G} := \mu_{\delta}^{k}|_{G} \quad \text{and} \quad \mu_{\delta}^{k + 1}|_{B} := \frac{\ell(Q)^{s}}{2 \cdot \mu_{\delta}^{k}(Q)} \cdot \mu_{\delta}^{k}|_{B}. \end{equation} 
Note that
\begin{equation}\label{form29} \mu_{\delta}^{k + 1}(Q) \leq \frac{\ell(Q)^{s}}{2} + \frac{\ell(Q)^{s}}{2 \cdot \mu_{\delta}^{k}(Q)} \cdot \mu_{\delta}^{k}(B) \leq \ell(Q)^{s}. \end{equation}
Moreover, since $\mu_{\delta}^{k}(G) \leq \ell(Q)^{s}/2 \leq \mu_{\delta}^{k}(Q)/2$ by \eqref{form74}, we have $\mu_{\delta}^{k}(B) \geq \mu_{\delta}^{k}(Q)/2$, and consequently
\begin{equation}\label{form28} \mu_{\delta}^{k + 1}(Q) \geq \frac{\ell(Q)^{s}}{4} \cdot \frac{2 \cdot \mu_{\delta}^{k}(B)}{\mu_{\delta}^{k}(Q)} \geq \frac{\ell(Q)^{s}}{4}. \end{equation}
We have now defined $\mu_{\delta}^{k + 1}$ on one cube $Q \in \calD_{2^{k + 1}\delta}$, and we repeat the same procedure on each of them. It is worth pointing out that
\begin{equation}\label{form40} \mu_{\delta}^{k + 1}(A) \leq \mu_{\delta}^{k}(A), \qquad A \subset \R^{n}, \: k \geq 0, \end{equation}
since $\ell(Q)^{s}/(2 \cdot \mu_{\delta}^{k}(Q)) < 1/2$ in \eqref{form30} (again by \eqref{form74}).

Let $N \geq 0$ be the index such that $2^{N}\delta = 1$, and set $\mu_{\delta} := \mu_{\delta}^{N}$. Then $\mu_{\delta}([0,1)^{n}) \leq 1$ by \eqref{form29} with $Q = [0,1)^{n}$, and since $\mu_{\delta}(\R^{n} \, \setminus \, [0,1)^{n}) = 0$, we also have 
\begin{equation}\label{form75} \mu_{\delta}(Q) \leq \ell(Q)^{s} \quad \text{for all } Q \text{ dyadic with } \ell(Q) \geq 1. \end{equation}

We next plan to check the bounds \eqref{upperBound}-\eqref{lowerBound} for the measure $\mu_{\delta}$. We start by verifying a version of \eqref{upperBound} for dyadic cubes: fix a cube $Q \in \calD_{2^{k}\delta}$ for some $k \geq 0$. If $k \geq N$, just recall \eqref{form75}. If $k < N$, the construction ensures that $\mu_{\delta}^{k}(Q) \leq \ell(Q)^{s}$, and then \eqref{form40} implies that $\mu_{\delta}(Q) \leq \mu_{\delta}^{k}(Q) \leq \ell(Q)^{s}$. So, we conclude that $\mu_{\delta}(Q) \leq \ell(Q)^{s}$ for all dyadic cubes of side-length $\geq \delta$. Since every ball $B(x,r)$, with $x \in \R^{n}$ and $r \geq \delta$, can be covered by $m \lesssim_{n} 1$ such dyadic cubes $Q_{1},\ldots,Q_{m}$ of side-lengths $\ell(Q_{j}) \in [r,2r)$, we infer that 
\begin{equation}\label{form34} \mu(B(x,r)) \lesssim_{n} r^{s}, \qquad x \in \R^{n}, \; r \geq \delta. \end{equation}

Next, we verify the lower bound \eqref{lowerBound} for the measure $\mu_{\delta}$, namely
\begin{equation}\label{form32} \mu_{\delta}(\overline{Q}) \geq \mu_{\delta}(Q) \geq a_{n} \min\{\calH^{s}_{\infty}(E \cap Q),|Q|\}, \qquad Q \in \calD_{2^{k}\delta}, \, 0 \leq k \leq k_{0}. \end{equation}
The constant $a_{n} > 0$ will only depend on $n$. So, fix 
\begin{displaymath} Q \in \calD_{2^{k_{0}}\delta} \quad \text{for some} \quad 0 \leq k_{0} \leq N. \end{displaymath}
We start by observing that
\begin{equation}\label{form31} \mu_{\delta}^{k_{0}}(Q) \geq b_{n} \cdot \calH^{s}_{\infty}(E \cap Q), \end{equation}
where $b_{n} > 0$ is another constant, to be determined in a moment. This follows from the fact that every point $x \in E_{\delta} \cap Q$ is contained in some maximal cube $Q_{x} \subset Q$ of side-length $\delta \leq \ell(Q_{x}) \leq \ell(Q)$ such that $\mu_{\delta}^{k_{0}}(Q_{x}) \geq \ell(Q_{x})^{s}/4$. Indeed, we may take $Q_{x}$ to be the largest cube satisfying $x \in Q_{x} \subset Q$, where alternative \eqref{form74} occurred up to step $k_{0}$ (or simply $Q_{x} \in \calD_{\delta}$ if \eqref{form74} does not occur before and including step $k_{0}$ on any cube containing $x$, because then $\mu^{k}_{\delta}(Q_{x}) = \mu^{0}_{\delta}(Q_{x}) = \delta^{s} = \ell(Q_{x})^{s}$). Then, denoting $0 \leq k \leq k_{0}$ the index such that $Q_{x} \in \calD_{2^{k}\delta}$, we see from \eqref{form28}, and the maximality of $Q_{x}$, that 
\begin{displaymath} \mu^{k_{0}}_{\delta}(Q_{x}) = \mu^{k}_{\delta}(Q_{x}) \geq \ell(Q_{x})^{s}/4. \end{displaymath}
Now, if $\mathfrak{m}(Q)$ is the collection of these maximal, hence disjoint, cubes $Q_{x} \subset Q$, we find that
\begin{displaymath} \mu_{\delta}^{k_{0}}(Q) = \sum_{Q' \in \mathfrak{m}(Q)} \mu_{\delta}^{k_{0}}(Q') \geq \tfrac{1}{4} \sum_{Q \in \mathfrak{m}(Q)} \ell(Q')^{s} \gtrsim_{n} \calH^{s}_{\infty}(E \cap Q), \end{displaymath} 
as claimed in \eqref{form31}, with constant $b_{n} := \tfrac{1}{4} \diam([0,1]^{n})^{-n}$.

We now claim that \eqref{form32} holds for some constant $a_{n} \gtrsim b_{n}$. To this end, assume that \eqref{form32} fails with constant $b_{n}$. By \eqref{form31} and \eqref{form40}, we certainly have $\mu_{\delta}^{k}(Q) \geq b_{n} \cdot \min\{\calH^{s}_{\infty}(E \cap Q),|Q|\}$ for $0 \leq k \leq k_{0}$, so the failure of \eqref{form32} means that there exists a first index $k_{0} < k_{1} \leq N$ such that
\begin{displaymath} \mu_{\delta}^{k_{1}}(Q) < b_{n} \cdot \min\{\calH^{s}_{\infty}(E \cap Q),|Q|\} \leq b_{n} \cdot |Q| < |Q|/2. \end{displaymath}
Since $k_{1}$ is the first index with this property, we conclude that the unique cube $Q_{1} \in \calD_{2^{k_{1}}\delta}$ containing $Q$ must satisfy alternative \eqref{form74} with $k = k_{1} - 1$ (otherwise $\mu_{\delta}^{k_{1}}(Q) = \mu_{\delta}^{k_{1} - 1}(Q)$). Then,
\begin{equation}\label{form33} \mu_{\delta}^{k_{1}}(Q) \geq \frac{\ell(Q_{1})^{s}}{2 \cdot \mu_{\delta}^{k_{1} - 1}(Q_{1})} \cdot \mu_{\delta}^{k_{1} - 1}(Q) \gtrsim_{n} b_{n} \cdot \min\{\calH^{s}_{\infty}(Q \cap E),|Q|\}, \end{equation}
using that $\mu_{\delta}^{k_{1} - 1}(Q_{1}) \lesssim_{n} \ell(Q_{1})^{s}$ (that is, even though $\mu_{\delta}^{k_{1} - 1}(Q_{1}) > \ell(Q_{1})^{s}$ by alternative \eqref{form74}, the converse inequality still "almost" holds, since it holds for the children of $Q_{1}$).

Now, for all indices $k_{1} \leq k < N$, the cube $Q$ will satisfy 
\begin{displaymath} \mu_{\delta}^{k}(Q) \leq \mu_{\delta}^{k_{1}}(Q) < |Q|/2, \end{displaymath}
and hence will be contained in the "good set" $G$ of step $k$ (associated with the particular cube of that step which happens to contain $Q$). Consequently, recalling \eqref{form30}, the value $k \mapsto \mu_{\delta}^{k}(Q)$ remains constant for $k_{1} \leq k \leq N$, and \eqref{form32} now follows from \eqref{form33}. 

The rest of the proof is carried out as in the usual proof of Frostman's lemma. After passing to a subsequence, the measures $\mu_{\delta}$ converge to a non-negative Radon measure supported on $E$ (noting that $\mu_{\delta}$ is supported on $\overline{E_{\delta}}$). The upper and lower bounds in \eqref{upperBound}-\eqref{lowerBound} follow from \eqref{form34} and \eqref{form32}, and standard results on weak convergence, see \cite[Theorem 1.24]{zbMATH01249699}. This completes the proof. \end{proof}

\bibliographystyle{plain}
\bibliography{references}

\end{document}